\documentclass[a4paper,12pt,intlimits,oneside]{amsart}

\usepackage{enumerate}
\usepackage{mathrsfs}
\usepackage{amsfonts,amsmath}

\usepackage{latexsym,amssymb,amsthm}

\newcommand{\comment}[1]{}

\newcommand{\R}{{\mathbb R}}
\newcommand{\pR}{{\mathbb R}\times\cdots\times {\mathbb R}}
\newcommand{\pint}{{\int_0^\infty}\cdots {\int_0^\infty}}
\newcommand{\pvarphi}{\varphi(t_1,\ldots,t_n)}
\newcommand{\pdt}{dt_1\ldots dt_n}
\newcommand{\pt}{t_1\ldots t_n}
\newcommand{\phardy}{H^1(\R\times\cdots\times\R)}

\def\H{{\mathcal H}}

\def\sign{{\rm {sign}\,}}

\newcounter{rea}
\newcounter{rek}
\newcounter{res}
\begin{document}
\title[]{The multi-parameter Hausdorff operators on $H^1$ and $L^p$}         
\author{Duong Quoc Huy} 
\address{Department of Natural Science and Technology, Tay Nguyen University, Daklak, Vietnam.}
\email{duongquochuy@ttn.edu.vn}
\author{Luong Dang Ky $^*$}
\address{Department of Mathematics, Quy Nhon University, 
170 An Duong Vuong, Quy Nhon, Binh Dinh, Viet Nam} 
\email{{\tt luongdangky@qnu.edu.vn}}
\thanks{This work is supported by Vietnam National Foundation for Science and Technology Development (Grant No. 101.02-2017.304)}
\thanks{$^*$Corresponding author}

\begin{abstract} 
	In the present paper, we characterize the  nonnegative  functions $\varphi$ for which the multi-parameter Hausdorff operator $\mathcal H_\varphi$ generated by $\varphi$ is bounded on the multi-parameter Hardy space $H^1(\mathbb R\times\cdots\times\mathbb R)$ or $L^p(\mathbb R^n)$, $p\in [1,\infty]$.  The corresponding operator norms are also obtained.  Our results improve some recent results in \cite{FZ, LM2, LM3, We} and give an answer to an open question posted by Liflyand \cite{Li}.
	
\end{abstract}

\maketitle
\newtheorem{theorem}{Theorem}[section]
\newtheorem{lemma}{Lemma}[section]
\newtheorem{proposition}{Proposition}[section]
\newtheorem{remark}{Remark}[section]
\newtheorem{corollary}{Corollary}[section]
\newtheorem{definition}{Definition}[section]
\newtheorem{example}{Example}[section]
\numberwithin{equation}{section}
\newtheorem{Theorem}{Theorem}[section]
\newtheorem{Lemma}{Lemma}[section]
\newtheorem{Proposition}{Proposition}[section]
\newtheorem{Remark}{Remark}[section]
\newtheorem{Corollary}{Corollary}[section]
\newtheorem{Definition}{Definition}[section]
\newtheorem{Example}{Example}[section]
\newtheorem*{theorema}{Theorem A}

\section{Introduction and main result} 
\allowdisplaybreaks

Let $\varphi$ be a locally integrable function on $(0,\infty)$. The classical one-parameter Hausdorff operator $\H_\varphi$ is defined for suitable functions $f$ on $\R$ by
$$\H_\varphi f(x)=\int_0^\infty f\left(\frac{x}{t}\right) \frac{\varphi(t)}{t} dt.$$
The Hausdorff operator $\H_\varphi$ is an interesting operator in harmonic analysis. There are many classical operators in analysis which are special cases of the Hausdorff operator if one chooses suitable kernel functions $\varphi$, such as the classical Hardy operator, its adjoint operator, the Ces\`aro type operators, the Riemann-Liouville fractional integral operator. See the survey article \cite{Li13} and the references therein. In the recent years, there is an increasing interest in the study of boundedness of the Hausdorff operator on some function spaces, see for example \cite{An, CFZ, FZ, HKQ1, HKQ2, Li, Li13, LM1, LM2, LM3, WLY, We, WC}.

When $\varphi$ is a locally integrable function on $(0,\infty)^n$, there are several high-dimensional extensions of $\H_\varphi$. One of them is the {\sl multi-parameter Hausdorff operator} $\H_\varphi$ defined for suitable functions $f$ on $\R^n$ by
$$\H_\varphi f(x_1,\ldots,x_n)=\int_0^\infty\cdots\int_0^\infty f\left(\frac{x_1}{t_1},\ldots,\frac{x_n}{t_n}\right) \frac{\varphi(t_1,\ldots,t_n)}{t_1\ldots t_n} dt_1\cdots dt_n.$$

Let $\Phi^{(1)}, \ldots, \Phi^{(n)}$ be $C^\infty$-functions with compact support satisfying $\int_{\R}\Phi^{(1)}(x)dx=\cdots= \int_{\R}\Phi^{(n)}(x)dx=1$. Then, for any $(t_1,\ldots,t_n)\in (0,\infty)^n$, we denote
$$\otimes_{j=1}^n \Phi^{(j)}_{t_j}({\bf x}):= \prod_{j=1}^n \frac{1}{t_j}\Phi^{(j)}\left(\frac{x_j}{t_j}\right),\quad {\bf x}=(x_1,\ldots,x_n)\in\R^n.$$

 Following Gundy and Stein \cite{GS}, we define the {\it multi-parameter Hardy space} $H^1(\R\times\cdots\times\R)$ as the set of all functions $f\in L^1(\R^n)$ such that 
$$\|f\|_{H^1(\R\times\cdots\times\R)} := \left\|M_{\Phi} f\right\|_{L^1(\R)}<\infty,$$
where $M_{\Phi} f$ is the {\it multi-parameter smooth maximal function} of $f$ defined by
$$M_{\Phi}f({\bf x})= \sup_{(t_1,\ldots,t_n)\in (0,\infty)^n}|f*(\otimes_{j=1}^n\Phi^{(j)}_{t_j})({\bf x})|,\quad {\bf x}\in\R^n.$$

\begin{remark}\label{H1 implies Hn}
	\begin{enumerate}[\rm (i)]
		\item $\|\cdot\|_{\phardy}$ defines a norm on $\phardy$, whose size depends on the choice of $\{\Phi^{(j)}\}_{j=1}^n$, but the space $\phardy$ does not depend on this choice.
		\item If $f$ is in $H^1(\R)$, then the function
		$$f\otimes\cdots\otimes f({\bf x})= \prod_{j=1}^n f(x_j),\quad {\bf x}=(x_1,\ldots,x_n)\in \R^n,$$
		is in $\phardy$. Moreover, there exist two positive constants $C_1, C_2$ independent of $f$ such that
		$$C_1\|f\|_{H^1(\R)}^n \leq \left\|f\otimes\cdots\otimes f\right\|_{\phardy}\leq C_2 \|f\|_{H^1(\R)}^n.$$
	\end{enumerate}	 
\end{remark}

In the setting of two-parameter, Liflyand and M\'oricz  showed  in \cite{LM2} that $\H_\varphi$ is bounded on $H^1(\R\times\R)$ provided $\varphi\in L^1((0,\infty)^2)$. In the setting of $n$-parameter, one of Weisz's important results (see \cite[Theorem 7]{We}) showed that $\H_\varphi$ is bounded on $\phardy$ provided $\varphi(t_1,\ldots,t_n)= \prod_{i=1}^n \varphi_i(t_i)$ with $\varphi_i\in L^1(\R)$ for all $1\leq i\leq n$. Recently, in the setting of two-parameter, Fan and Zhao showed in \cite{FZ} that the condition $\varphi\in L^1((0,\infty)^2)$  is also a necessary condition for $H^1(\R\times\R)$-boundedness of $\H_\varphi$ if $\varphi$ is nonnegative valued. However, it seems that Fan-Zhao's method can not be used to obtain the exact norm of $\H_\varphi$ on $H^1(\R\times\R)$. So, in the setting of $n$-parameter, a natural question arises: Can one find the exact norm of $\H_\varphi$ on $\phardy$? Very recently, in the setting of one-parameter, this question was solved by Hung, Ky and Quang \cite{HKQ1}.

Motivated by the above question and an open question posted by Liflyand \cite[Problem 5]{Li}, we characterize the  nonnegative  functions $\varphi$ for which $\H_\varphi$  is bounded on  $\phardy$. More precisely, our main result is the following:

\begin{theorem}\label{main theorem}
	Let $\varphi$ be a nonnegative function in $L^1_{\rm loc}((0,\infty)^n)$. Then $\H_\varphi$ is bounded on $H^1(\R\times\cdots\times\R)$ if and only if
	\begin{equation}\label{main inequality}
	\int_0^\infty\cdots \int_0^\infty \varphi(t_1,\ldots,t_n)dt_1\cdots dt_n<\infty.
	\end{equation}
	Moreover, in that case,	
	$$\|\H_\varphi\|_{H^1(\R\times\cdots\times\R)\to H^1(\R\times\cdots\times\R)}= \int_0^\infty\cdots \int_0^\infty \varphi(t_1,\ldots,t_n)dt_1\cdots dt_n.$$
\end{theorem}

Theorem \ref{main theorem} not only gives an affirmative answer to the above question, but also gives an answer to \cite[Problem 5]{Li}. It should be pointed out that  the norm of the Hausdorff operator $\H_\varphi$ ($\pint\pvarphi\pdt$) does not depend on the choice of the above functions $\{\Phi^{(j)}\}_{j=1}^n$, moreover, it still holds when the above norm $\|\cdot\|_{\phardy}$ is replaced by
$$\|f\|_{\phardy}:=\sum_{{\bf e}\in \{0,1\}^n} \|{\bf H_e} f\|_{L^1(\R^n)},$$
where ${\bf H_e} f$'s are the {\sl multi-parameter Hilbert transforms} of $f$. See Theorem \ref{replace by the Hilbert transforms} for details.

Also we characterize the  nonnegative  functions $\varphi$ for which $\H_\varphi$  is bounded on $L^p(\R^n)$, $p\in [1,\infty]$. Our next result can be stated as follows.

\begin{theorem}\label{sharp constants for Lp}
	Let $p\in [1,\infty]$ and let $\varphi$ be a nonnegative function in $L^1_{\rm loc}((0,\infty)^n)$. Then $\H_\varphi$ is bounded on $L^p(\R^n)$ if and only if
	\begin{equation}\label{main inequality 1}
	\int_0^\infty\cdots \int_0^\infty \frac{\varphi(t_1,\ldots,t_n)}{t_1^{1-1/p}\ldots t_n^{1-1/p}}dt_1\cdots dt_n<\infty.
	\end{equation}
	Moreover, in that case,
	$$
	\|\H_\varphi\|_{L^p(\R^n)\to L^p(\R^n)}=\int_0^\infty\cdots \int_0^\infty \frac{\varphi(t_1,\ldots,t_n)}{t_1^{1-1/p}\ldots t_n^{1-1/p}}dt_1\cdots dt_n.
	$$
\end{theorem}

Throughout the whole article, {\sl we always assume that $\varphi$ is a nonnegative function in $L^1_{\rm loc}((0,\infty)^n)$} and denote by $C$ a positive constant which is independent of the main parameters, but it may vary from line to line. The symbol $A\lesssim B$ means that $A\leq C B$. If $A\lesssim B$ and $B\lesssim A$, then we write $A\sim B$.


\section{Norm of $\H_\varphi$  on $L^p(\R^n)$}

The main purpose of this section is to give the proof of Theorem \ref{sharp constants for Lp}. Let us first consider the operator  $\H^*_\varphi$ defined by
$$\H^*_\varphi f(x_1,\ldots,x_n)=\pint f(t_1 x_1,\ldots,t_n x_n) \pvarphi \pdt.$$
Studying this operator on the spaces $L^p(\R^n)$ is useful in proving the main theorem (Theorem \ref{main theorem}) in the next section.

Remark that $\H^*_\varphi= \H_{\overline{\varphi}}$ with $\overline{\varphi}({\bf t})=\frac{\varphi(1/t_1,\ldots,1/t_n)}{t_1\ldots t_n}$ for all ${\bf t}=(t_1,\ldots,t_n)\in (0,\infty)^n$. Hence, by Theorems \ref{main theorem} and \ref{sharp constants for Lp}, we obtain:

\begin{theorem}\label{adjoint operator, hardy}
	$\H^*_\varphi$ is bounded on $H^1(\R\times\cdots\times\R)$ if and only if
	\begin{equation}\label{an inequality for adjoint operator, hardy}
	\int_0^\infty\cdots \int_0^\infty \frac{\varphi(t_1,\ldots,t_n)}{t_1\ldots t_n}dt_1\cdots dt_n<\infty.
	\end{equation}
	Moreover, in that case,	
	$$\|\H^*_\varphi\|_{H^1(\R\times\cdots\times\R)\to H^1(\R\times\cdots\times\R)}= \int_0^\infty\cdots \int_0^\infty \frac{\varphi(t_1,\ldots,t_n)}{t_1\ldots t_n}dt_1\cdots dt_n.$$
\end{theorem}

\begin{theorem}\label{adjoint operator, sharp constants for Lp}
	Let $p\in [1,\infty]$. Then $\H^*_\varphi$ is bounded on $L^p(\R^n)$ if and only if
	\begin{equation}\label{an inequality for adjoint operator, Lp}
	\int_0^\infty\cdots \int_0^\infty \frac{\varphi(t_1,\ldots,t_n)}{t_1^{1/p}\ldots t_n^{1/p}}dt_1\cdots dt_n<\infty.
	\end{equation}
	Moreover, in that case,
	$$
	\|\H^*_\varphi\|_{L^p(\R^n)\to L^p(\R^n)}=\int_0^\infty\cdots \int_0^\infty \frac{\varphi(t_1,\ldots,t_n)}{t_1^{1/p}\ldots t_n^{1/p}}dt_1\cdots dt_n.
	$$
\end{theorem}

By Theorems \ref{sharp constants for Lp}, \ref{adjoint operator, sharp constants for Lp} and the Fubini theorem, $\H^*_\varphi$ can be viewed
as the Banach space adjoint of $\H_\varphi$ and vice versa. More precisely, we have:

\begin{theorem}\label{adjoint operators on Lp}
	Let $p\in [1,\infty]$ and $1/p'+1/p=1$. 
	\begin{enumerate}[\rm (i)]
		\item If (\ref{main inequality 1}) holds, then, for all $f\in L^p(\R^n)$ and all $g\in L^{p'}(\R^n)$,
		$$\int_{\R^n} \H_\varphi f({\bf x}) g({\bf x})d{\bf x}=\int_{\R^n} f({\bf x})\H^*_\varphi g({\bf x}) d{\bf x}.$$
		\item If (\ref{an inequality for adjoint operator, Lp}) holds, then, for all $f\in L^p(\R^n)$ and all $g\in L^{p'}(\R^n)$,
		$$\int_{\R^n} \H^*_\varphi f({\bf x}) g({\bf x})d{\bf x}=\int_{\R^n} f({\bf x})\H_\varphi g({\bf x}) d{\bf x}.$$
	\end{enumerate}
\end{theorem}

As a consequence of the above theorem, we get the following.

\begin{corollary}\label{relation to the Fourier transform}
	Let $p\in [1,2]$. 
	\begin{enumerate}[\rm (i)]
		\item If (\ref{main inequality 1}) holds, then, for all $f\in L^p(\R^n)$,
		$$\widehat{\H_\varphi f}= \H_\varphi^*\hat{f}.$$
		\item If (\ref{an inequality for adjoint operator, Lp}) holds, then, for all $f\in L^p(\R^n)$,
		$$\widehat{\H^*_\varphi f}= \H_\varphi\hat{f}.$$
	\end{enumerate}
\end{corollary}

\begin{proof}
	We prove only (i) since the proof of (ii) is similar. Moreover, from the Hausdorff-Young theorem and the fact that $L^1(\R^n)\cap L^p(\R^n)$ is dense in $L^p(\R^n)$, we consider only the case $p=1$. For all ${\bf y}=(y_1,\ldots,y_n)\in\R^n$, by Theorem \ref{adjoint operators on Lp}(i) and the Fubini theorem, we get 
	\begin{eqnarray*}
		\widehat{\H_\varphi f}({\bf y}) &=& \int_{\R^n} \H_\varphi f({\bf x}) e^{-2\pi i {\bf x}\cdot{\bf y}}d{\bf x}\\
		&=& \int_{\R^n} f({\bf x}) d{\bf x} \pint e^{-2\pi i\sum_{j=1}^n t_j x_j y_j}\pvarphi \pdt\\
		&=& \pint \hat{f}(t_1 y_1,\ldots,t_n y_n) \pvarphi\pdt\\
		&=& \H^*_\varphi \hat f ({\bf y}).
	\end{eqnarray*}
	This completes the proof of Corollary \ref{relation to the Fourier transform}.
	
\end{proof}


\begin{proof}[\bf Proof of Theorem \ref{sharp constants for Lp}]
	Since the case $p=\infty$ is trivial, we consider only the case $p\in [1,\infty)$. Suppose that (\ref{main inequality 1}) holds. For any $f\in L^p(\R^n)$, by the Minkowski inequality, we obtain
	\begin{eqnarray*}
		\|\H_\varphi f\|_{L^p(\R^n)} &\leq& \int_0^\infty\cdots \int_0^\infty \left\|f\left(\frac{\cdot}{t_1},\ldots,\frac{\cdot}{t_n}\right)\right\|_{L^p(\R^n)} \frac{\varphi(t_1,\ldots,t_n)}{t_1\ldots t_n} dt_1\cdots dt_n\\
		&=& \|f\|_{L^p(\R^n)} \int_0^\infty\cdots \int_0^\infty \frac{\varphi(t_1,\ldots,t_n)}{t_1^{1-1/p}\ldots t_n^{1-1/p}}dt_1\cdots dt_n.
	\end{eqnarray*}
	This proves that $\H_\varphi$ is bounded on $L^p(\R^n)$, moreover,
	\begin{equation}\label{sharp constants for Lp, 1}
	\|\H_\varphi\|_{L^p(\R^n)\to L^p(\R^n)}\leq \int_0^\infty\cdots \int_0^\infty \frac{\varphi(t_1,\ldots,t_n)}{t_1^{1-1/p}\ldots t_n^{1-1/p}}dt_1\cdots dt_n.	
	\end{equation}

	Conversely, suppose that $\H_\varphi$ is bounded on $L^p(\R^n)$. For any $\varepsilon>0$, take 
	$$f_\varepsilon({\bf x})= \prod_{j=1}^n |x_j|^{-1/p-\varepsilon}\chi_{\{y_j\in\R: |y_j|\geq 1\}}(x_j)$$
	for all ${\bf x}=(x_1,\ldots,x_n)\in\R^n.$ Then, it is easy to see that $f_\varepsilon\in L^p(\R^n)$ and
	$$\H_\varphi f_\varepsilon({\bf x})= \prod_{j=1}^n |x_j|^{-1/p-\varepsilon}\int_{0}^{|x_1|}dt_1\cdots \int_{0}^{|x_{n-1}|}dt_{n-1}\int_{0}^{|x_n|}\frac{\varphi(t_1,\ldots,t_n)}{t_1^{1-1/p-\varepsilon}\ldots t_n^{1-1/p-\varepsilon}}dt_n$$
	for all ${\bf x}=(x_1,\ldots,x_n)\in\R^n.$ Some simple computations give
	\begin{eqnarray*}
		\|\H_\varphi f_\varepsilon\|_{L^p(\R^n)} &\geq& \int_{0}^{1/\varepsilon}\cdots \int_{0}^{1/\varepsilon} \frac{\varphi(t_1,\ldots,t_n)}{t_1^{1-1/p-\varepsilon}\ldots t_n^{1-1/p-\varepsilon}}dt_1\cdots dt_n \times\\
		&&\times\left(\prod_{j=1}^n\int_{\{x_j\in\R:|x_j|\geq 1/\varepsilon\}}   |x_j|^{-1-p\varepsilon} dx_j\right)^{1/p}\\
		&=& \int_{0}^{1/\varepsilon}\cdots \int_{0}^{1/\varepsilon} \frac{\varphi(t_1,\ldots,t_n)}{t_1^{1-1/p-\varepsilon}\ldots t_n^{1-1/p-\varepsilon}}dt_1\cdots dt_n \left(\varepsilon^{n\varepsilon}\|f_\varepsilon\|_{L^p(\R^n)}\right).
	\end{eqnarray*}
	Therefore,
	\begin{eqnarray*}
		\|\H_\varphi\|_{L^p(\R^n)\to L^p(\R^n)} &\geq& \frac{\|\H_\varphi f_\varepsilon\|_{L^p(\R^n)}}{\|f_\varepsilon\|_{L^p(\R^n)}}\\
		&\geq& \varepsilon^{n\varepsilon} \int_{0}^{1/\varepsilon}\cdots \int_{0}^{1/\varepsilon} \frac{\varphi(t_1,\ldots,t_n)}{t_1^{1-1/p-\varepsilon}\ldots t_n^{1-1/p-\varepsilon}}dt_1\cdots dt_n.
	\end{eqnarray*}
	Letting $\varepsilon\to 0$, we obtain
	$$\|\H_\varphi\|_{L^p(\R^n)\to L^p(\R^n)}\geq \int_{0}^{\infty}\cdots \int_{0}^{\infty} \frac{\varphi(t_1,\ldots,t_n)}{t_1^{1-1/p}\ldots t_n^{1-1/p}}dt_1\cdots dt_n.$$
	This, together (\ref{sharp constants for Lp, 1}), implies that
	$$\|\H_\varphi\|_{L^p(\R^n)\to L^p(\R^n)} = \int_{0}^{\infty}\cdots \int_{0}^{\infty} \frac{\varphi(t_1,\ldots,t_n)}{t_1^{1-1/p}\ldots t_n^{1-1/p}}dt_1\cdots dt_n,$$
	and thus ends the proof of Theorem \ref{sharp constants for Lp}.
	
\end{proof}


\section{Norm of $\H_\varphi$ on $\phardy$}

The main purpose of this section is to give the proof of Theorem \ref{main theorem} and to show that the norm of the Hausdorff operator $\H_\varphi$ in Theorem \ref{main theorem}  still holds when one replaces the norm $\|\cdot\|_{\phardy}$ by the norm $\|\cdot\|_*$ (see (\ref{norm via Hilbert transforms}) below).

Let $\mathbb C_+^n$ be the upper half-plan in $\mathbb C^n$, that is,
$$\mathbb C_+^n=\prod_{j=1}^n \{z_j=x_j+iy_j\in\mathbb C: y_j>0\}.$$
Following Gundy-Stein \cite{GS} and Lacey \cite{La}, a function $F: \mathbb C_+^n \to \mathbb C$ is said to be in the Hardy space $\mathcal H^1_a(\mathbb C_+^n)$ if it is holomorphic in each variable separately and
$$\|F\|_{\mathcal H^1_a(\mathbb C_+^n)}:= \sup_{(y_1,\ldots,y_n)\in (0,\infty)^n}\int_{-\infty}^{\infty}\cdots \int_{-\infty}^{\infty} |F(x_1+i y_1,\ldots,x_n+iy_n)|dx_1\ldots dx_n<\infty.$$

Let $j\in \{1,\ldots,n\}$. For any $f\in L^1(\R^n)$, the Hilbert transform  $H_j f$ computed in the $j^{\rm th}$ variable is defined by
$$H_j f({\bf x})=\frac{1}{\pi}{\rm p.v.} \int_{-\infty}^\infty \frac{f(x_1,\ldots, x_j-y,\ldots,x_n)}{y}dy.$$

For any ${\bf e}=(e_1,\ldots,e_n)\in \mathbb E:= \{0,1\}^n$, denote
$${\bf H}_{\bf e}=\prod_{j=1}^n H_{j}^{e_j}$$
with $H_{j}^{e_j}=I$ for $e_j=0$ while $H_{j}^{e_j}=H_j$ for $e_j=1$.

The following two theorems are well-known, see for example  \cite{GS, La, LPPW, We}.

\begin{theorem}\label{equivalent characterizations of H1}
A function $f$ is in $\phardy$ if and only if ${\bf H_e} f$ is in $L^1(\R^n)$ for all ${\bf e}\in \mathbb E$.	Moreover, in that case,
	$$\|f\|_{\phardy}\sim \sum_{{\bf e}\in \mathbb E} \|{\bf H_e} f\|_{L^1(\R^n)}.$$
\end{theorem}

\begin{theorem}\label{boundary value function}
	Let $F\in \mathcal H^1_a(\mathbb C_+^n)$. Then the boundary value function $f$ of $F$, which is defined by
	$$f(x_1,\ldots,x_n)=\lim_{(y_1,\ldots,y_n)\to (0,\ldots,0)} F(x_1+iy_1,\ldots,x_n+iy_n),$$
	a. e. ${\bf x}=(x_1,\ldots,x_n)\in\R^n$, is in $\phardy$. Moreover,
	$$\|f\|_{\phardy}\sim \|f\|_{L^1(\R^n)}=\|F\|_{\mathcal H^1_a(\mathbb C_+^n)}$$
	and, for all ${\bf x}+i{\bf y}=(x_1+iy_1,\ldots,x_n+iy_n)\in \mathbb C_+^n$,
	\begin{eqnarray*}
	F({\bf x}+i{\bf y})	&=& \int_{\R}\cdots\int_{\R} f(x_1-u_1,\ldots,x_n-u_n)\prod_{j=1}^n \frac{1}{y_j}P\left(\frac{u_j}{y_j}\right) du_1\ldots du_n\\
	&=:& f*(\otimes_{j=1}^n P_{y_j})({\bf x}),
	\end{eqnarray*}
	where $P(u)= \frac{1}{1+u^2}, u\in\R$, is the Poisson kernel on $\R$.	
\end{theorem}

In order to prove Theorem \ref{main theorem}, we also need the following two lemmas.

\begin{lemma}\label{the necessary condition}
	Let $\varphi$ be such that $\H_\varphi$ is bounded from  $\phardy$ into $L^1(\R^n)$. Then (\ref{main inequality}) holds.
\end{lemma}

\begin{lemma}\label{key lemma}
	Let $\varphi$ be such that (\ref{main inequality}) holds. Then:
	\begin{enumerate}[\rm (i)]
		\item $\H_\varphi$ is bounded on $H^1(\pR)$, moreover,
		$$\|\H_\varphi\|_{H^1(\pR)\to H^1(\pR)}\leq \pint \pvarphi \pdt.$$
		\item If supp $\varphi\subset [0,1]^n$, then
		$$\|\H_\varphi\|_{H^1(\pR)\to H^1(\pR)}= \int_0^1\cdots \int_0^1 \pvarphi \pdt.$$
	\end{enumerate}
\end{lemma}

\begin{proof}[Proof of Lemma \ref{the necessary condition}]
	Since the function
	$$f(x)=\frac{x}{(1+x^2)^2},\quad x\in\R,$$
	is in $H^1(\R)$ (see \cite[Theorem 3.3]{HKQ1}), Remark \ref{H1 implies Hn}(ii) yields that
	$$f\otimes\cdots\otimes f({\bf x})=\prod_{j=1}^n \frac{x_j}{(1+x_j^2)^2},\quad {\bf x}=(x_1,\ldots,x_n)\in\R^n,$$
	is in $H^1(\R\times\cdots\times\R)$. Hence, the function
	$$\H_\varphi\left(f\otimes\cdots\otimes f\right)({\bf x})= \int_0^\infty\cdots \int_0^\infty \prod_{j=1}^n \frac{\frac{x_j}{t_j}}{\left[1+\left(\frac{x_j}{t_j}\right)^2\right]^2} \frac{\varphi(t_1,\ldots,t_n)}{t_1\ldots t_n}dt_1\cdots dt_n,$$
	${\bf x}=(x_1,\ldots,x_n)\in\R^n$,	is in $L^1(\R^n)$ since $\H_\varphi$ is bounded from $H^1(\R\times\cdots\times\R)$ into $L^1(\R^n)$. As a consequence,
	\begin{eqnarray*}
	 &&\left[\int_0^\infty\frac{y}{(1+y^2)^2}dy\right]^n \int_0^\infty\cdots \int_0^\infty \varphi(t_1,\ldots,t_n)dt_1\cdots dt_n\\
	 &=& \int_{[0,\infty)^n}d{\bf x}\int_0^\infty\cdots \int_0^\infty \prod_{j=1}^n \frac{\frac{x_j}{t_j}}{\left[1+\left(\frac{x_j}{t_j}\right)^2\right]^2} \frac{\varphi(t_1,\ldots,t_n)}{t_1\ldots t_n}dt_1\cdots dt_n\\
	 &\leq& \|\H_\varphi\left(f\otimes\cdots\otimes f\right)\|_{L^1(\R^n)} <\infty
	\end{eqnarray*}
	which proves  (\ref{main inequality}), and thus ends the proof of Lemma \ref{the necessary condition}.
	
\end{proof}

\begin{proof}[Proof of Lemma \ref{key lemma}]
	(i)	For any $f\in \phardy$, by the Fubini theorem, 
	\begin{eqnarray*}
		&& M_\Phi(\H_\varphi f)({\bf x}) \\
		&=& \sup_{(r_1,\ldots,r_n)\in (0,\infty)^n} \left|\int_{\R^n} d{\bf y} \pint (\otimes_{j=1}^n \Phi^{(j)}_{r_j})({\bf x}-{\bf y}) f\left(\frac{y_1}{t_1},\ldots,\frac{y_n}{t_n}\right)\frac{\pvarphi}{\pt}\pdt\right|\\
		&=& \sup_{(r_1,\ldots,r_n)\in (0,\infty)^n}\left|\pint \left(f*(\otimes_{j=1}^n \Phi^{(j)}_{r_j/t_j})\right)\left(\frac{x_1}{t_1},\ldots,\frac{x_n}{t_n}\right)\frac{\pvarphi}{\pt}\pdt\right|\\
		&\leq&  \H_\varphi(M_\Phi f)({\bf x})
	\end{eqnarray*}
	for all ${\bf x}=(x_1,\ldots,x_n)\in \R^n$. Hence, by Theorem \ref{sharp constants for Lp},
	\begin{eqnarray*}
		\|\H_\varphi f\|_{\phardy} &=&\|M_\Phi(\H_\varphi f)\|_{L^1(\R^n)}\\
		&\leq& \|\H_\varphi(M_\Phi f)\|_{L^1(\R^n)}\\
		&\leq& \pint \pvarphi \pdt \|M_\Phi f\|_{L^1(\R^n)}\\
		&=& \pint \pvarphi \pdt  \|f\|_{\phardy}.
	\end{eqnarray*}
	This proves that $\H_\varphi$ is bounded on $\phardy$, moreover,
	\begin{equation}\label{key lemma, 1}
		\|\H_\varphi\|_{\phardy\to \phardy}\leq \pint \pvarphi \pdt.
	\end{equation}	
	(ii)	Let $\delta\in (0,1)$ be arbitrary. Set $\varphi_\delta({\bf t}):= \varphi({\bf t})\chi_{[\delta,1]^n}({\bf t})$ for all ${\bf t}\in (0,\infty)^n$. Then, by (\ref{key lemma, 1}), we see that
	\begin{eqnarray*}
	 \|\H_{\varphi_\delta}\|_{\phardy \to \phardy}&\leq& \pint \varphi_\delta(t_1,\ldots,t_n)\pdt \\
	 &=& \int_{\delta}^{1}\cdots \int_{\delta}^{1} \pvarphi\pdt<\infty
	\end{eqnarray*}
	and 
	\begin{eqnarray}\label{key lemma, 2}
		&&\|\H_{\varphi}- \H_{\varphi_\delta}\|_{\phardy \to \phardy}\\\nonumber
		&\leq& \pint[\pvarphi- \varphi_\delta(t_1,\ldots,t_n)]\pdt\\\nonumber
		&=&  \int_{(0,1]^n\setminus [\delta,1]^n} \varphi({\bf t})d{\bf t}<\infty.\nonumber
	\end{eqnarray}

	For any $\varepsilon>0$, we define the function $F_\varepsilon:\mathbb C_+^n\to\mathbb C$ by
	$$F_\varepsilon(z_1,\ldots,z_n)=\prod_{j=1}^n\frac{1}{(z_j+i)^{1+\varepsilon}}$$
	where $\zeta^{1+\varepsilon}= |\zeta|^{1+\varepsilon} e^{i(1+\varepsilon)\arg \zeta}$ for all $\zeta\in\mathbb C$. Denote by $f_\varepsilon$ the boundary value function of $F_\varepsilon$, that is, $f_\varepsilon({\bf x})=\lim_{{\bf y}\to 0} F_\varepsilon({\bf x}+i{\bf y})$. Then, by Theorem \ref{boundary value function},
	\begin{equation}\label{key lemma, 3}
		\|f_\varepsilon\|_{\phardy}\sim \|F_\varepsilon\|_{\H^1_a(\mathbb C^n_+)}= \left[\int_{-\infty}^{\infty} \frac{1}{\sqrt{x^2+1}^{1+\varepsilon}}dx\right]^n <\infty,
	\end{equation}
	where the constants are independent of $\varepsilon$.

	For all ${\bf z}={\bf x} +i{\bf y}=(x_1+iy_1,\ldots,x_n+iy_n)=(z_1,\ldots,z_n)\in \mathbb C^n_+$, by the Fubini theorem and Theorem \ref{boundary value function}, we get
	\begin{eqnarray*}
		&& \left(\H_{\varphi_\delta}(f_\varepsilon)- f_{\varepsilon} \int_{(0,\infty)^n} \varphi_\delta({\bf t})d{\bf t}\right)*(\otimes_{j=1}^n P_{y_j})({\bf x})\\
		&=& \pint \prod_{j=1}^n\frac{1}{(\frac{z_j}{t_j}+i)^{1+\varepsilon}}\frac{\varphi_\delta(t_1,\ldots,t_n)}{\pt}\pdt-\\
		&&- \prod_{j=1}^n\frac{1}{(z_j+i)^{1+\varepsilon}}\pint\varphi_\delta(t_1,\ldots,t_n)\pdt \\
		&=&\int_{\delta}^{1}\cdots \int_{\delta}^{1} [\phi_{\varepsilon,{\bf z}}(t_1,\ldots,t_n)- \phi_{\varepsilon,{\bf z}}(1,\ldots,1)]\pvarphi \pdt,
	\end{eqnarray*}
	where $\phi_{\varepsilon,{\bf z}}(t_1,\ldots,t_n):= \prod_{j=1}^n \frac{t_j^\varepsilon}{(z_j+it_j)^{1+\varepsilon}}$. For any ${\bf t}=(t_1,\ldots,t_n)\in [\delta,1]^n$, a simple calculus gives
	\begin{eqnarray*}
		&&|\phi_{\varepsilon,{\bf z}}(t_1,\ldots,t_n)- \phi_{\varepsilon,{\bf z}}(1,\ldots,1)| \\
		&\leq& \sup_{s\in [0,1]} \sum_{j=1}^n |t_j-1| \left|\frac{\partial \phi_{\varepsilon,{\bf z}}}{\partial t_j}(t_j+ s(1-t_j))\right| \\
		&\leq& \sum_{j=1}^n \left(\frac{\varepsilon \delta^{-2}}{\sqrt{x_j^2+1}^{1+\varepsilon}} + \frac{(1+\varepsilon) \delta^{-2}}{\sqrt{x_j^2+1}^{2+\varepsilon}}\right)\prod_{\substack{k=1\\ k\ne j}}^n \frac{\delta^{-1}}{\sqrt{x_k^2+1}^{1+\varepsilon}}.
	\end{eqnarray*}
	Therefore, by Theorem \ref{boundary value function} again, 
	\begin{eqnarray*}
		&&\left\|\H_{\varphi_\delta}(f_\varepsilon)- f_{\varepsilon} \int_{(0,\infty)^n} \varphi_\delta({\bf t})d{\bf t}\right\|_{\phardy}\\
		 &\lesssim&  \left\|\sup_{(y_1,\ldots,y_n)\in (0,\infty)^n}\left(\H_{\varphi_\delta}(f_\varepsilon)- f_{\varepsilon} \int_{(0,\infty)^n} \varphi_\delta({\bf t})d{\bf t}\right)*(\otimes_{j=1}^n P_{y_j})\right\|_{L^1(\R^n)}\\
		&\leq&\int_{\delta}^{1}\cdots \int_{\delta}^{1} \pvarphi\pdt \times\\
		&&\times\sum_{j=1}^n \int_{-\infty}^{\infty}\cdots \int_{-\infty}^{\infty} \left(\frac{\varepsilon \delta^{-2}}{\sqrt{x_j^2+1}^{1+\varepsilon}} + \frac{(1+\varepsilon) \delta^{-2}}{\sqrt{x_j^2+1}^{2+\varepsilon}}\right)\prod_{\substack{k=1\\ k\ne j}}^n \frac{\delta^{-1}}{\sqrt{x_k^2+1}^{1+\varepsilon}} dx_1\ldots dx_n.
	\end{eqnarray*}
	This, together with (\ref{key lemma, 3}), yields
	\begin{eqnarray}\label{an useful estimate}
		&&\frac{\left\|\H_{\varphi_\delta}(f_\varepsilon)- f_{\varepsilon} \int_{(0,\infty)^n} \varphi_\delta({\bf t})d{\bf t}\right\|_{\phardy}}{\|f_\varepsilon\|_{\phardy}}\\
		&\lesssim& \int_{\delta}^{1}\cdots \int_{\delta}^{1} \pvarphi\pdt \times \nonumber\\
		&&\times \sum_{j=1}^n \frac{\delta^{1-n}\left[\varepsilon \delta^{-2} \int_{-\infty}^{\infty}\frac{1}{\sqrt{x_j^2+1}^{1+\varepsilon}}dx_j + (1+\varepsilon)\delta^{-2}\int_{-\infty}^{\infty}\frac{1}{\sqrt{x_j^2+1}^{2+\varepsilon}}dx_j\right]}{\int_{-\infty}^{\infty} \frac{1}{\sqrt{x_j^2+1}^{1+\varepsilon}}dx_j}   \nonumber\\
		&\lesssim& \int_{\delta}^{1}\cdots \int_{\delta}^{1} \pvarphi\pdt \times\nonumber\\
		&&\times\sum_{j=1}^n \left[\varepsilon \delta^{-1-n}+ \frac{(1+\varepsilon)\delta^{-1-n}\int_{-\infty}^{\infty} \frac{1}{x_j^2+1}dx_j}{\int_{-\infty}^{\infty} \frac{1}{\sqrt{x_j^2+1}^{1+\varepsilon}}dx_j}\right] \to 0 \nonumber
	\end{eqnarray}
	as $\varepsilon \to 0$. As a consequence,
	\begin{eqnarray*}
		\int_\delta^1 \cdots \int_\delta^1 \pvarphi \pdt &=& \int_{(0,\infty)^n} \varphi_\delta({\bf t})d{\bf t} \\
		&\leq& \|\H_{\varphi_\delta}\|_{\phardy\to \phardy}.
	\end{eqnarray*}
	This, together with (\ref{key lemma, 2}), allows us to conclude that
	$$\|\H_\varphi\|_{\phardy\to \phardy} \geq \int_0^1\cdots \int_0^1 \pvarphi\pdt$$
	since $\lim_{\delta\to 0}\int_{(0,1]^n\setminus [\delta,1]^n}\varphi({\bf t})d{\bf t}=0$. Hence, by (\ref{key lemma, 1}),
	$$\|\H_\varphi\|_{\phardy\to \phardy}= \int_0^1\cdots \int_0^1 \pvarphi\pdt.$$
	This completes the proof of Lemma \ref{key lemma}.

\end{proof}

Now we are ready to give the proof of Theorem \ref{main theorem}.

\begin{proof}[\bf Proof of Theorem \ref{main theorem}]
	By Lemma \ref{key lemma}(i), it suffices to prove that
	\begin{equation}\label{a lower bound}
		\pint \pvarphi\pdt\leq \|\H_\varphi\|_{\phardy\to \phardy}
	\end{equation}
	provided $\H_\varphi$ is bounded on $\phardy$. 	Indeed, by Lemma \ref{the necessary condition}, we have
	$$\pint \pvarphi\pdt<\infty.$$
	
	For any $m>0$, set $\varphi_m({\bf t}):= \varphi(m{\bf t})\chi_{(0,1)^n}({\bf t})$. Then, by Lemma \ref{key lemma}(i), we see that
	\begin{eqnarray}\label{truncation}
		&&\left\|\H_\varphi - \H_{\varphi_m\left(\frac{\cdot}{m}\right)}\right\|_{\phardy\to\phardy} \\
		&=& \left\|\H_{\varphi-\varphi_m\left(\frac{\cdot}{m}\right)}\right\|_{\phardy\to\phardy}\nonumber\\
		&\leq& \pint \left[\pvarphi-\varphi_m\left(\frac{t_1}{m},\ldots,\frac{t_n}{m}\right)\right]\pdt\nonumber\\
		&=& \int_{(0,\infty)^n\setminus (0,m)^n}\varphi({\bf t}) d{\bf t}.\nonumber
	\end{eqnarray}
Noting that
$$\left\|f\left(\frac{\cdot}{m}\right)\right\|_{\phardy}= m^n \|f(\cdot)\|_{\phardy}\quad\mbox{and}\quad \H_{\varphi_m\left(\frac{\cdot}{m}\right)}=\H_{\varphi_m}f\left(\frac{\cdot}{m}\right)$$
for all $f\in \phardy$, Lemma \ref{key lemma}(ii) gives
\begin{eqnarray*}
	\left\|\H_{\varphi_m\left(\frac{\cdot}{m}\right)}\right\|_{\phardy\to\phardy} &=& m^n \left\|\H_{\varphi_m}\right\|_{\phardy\to\phardy}\\
	&=& m^n \int_0^1\cdots\int_0^1 \varphi_m(t_1,\ldots,t_n)\pdt\\
	&=& \int_0^m\cdots\int_0^m \pvarphi \pdt.
\end{eqnarray*}
Combining this with (\ref{truncation})  allow us to conclude that
$$\|\H_\varphi\|_{\phardy\to\phardy}\geq \pint\pvarphi\pdt$$
since $\lim_{m\to\infty}\int_{(0,\infty)^n\setminus (0,m)^n}\varphi({\bf t}) d{\bf t}=0$. This proves (\ref{a lower bound}), and thus ends the proof of Theorem \ref{main theorem}.	
\end{proof}

From Theorem \ref{equivalent characterizations of H1}, one can define $\phardy$ as the space of functions $f\in L^1(\R^n)$ such that
\begin{equation}\label{norm via Hilbert transforms}
\|f\|_*:= \sum_{{\bf e}\in \mathbb E} \|{\bf H_e} f\|_{L^1(\R^n)}<\infty.
\end{equation}

Our last result is the following:

\begin{theorem}\label{replace by the Hilbert transforms}
	$\H_\varphi$ is bounded on $(H^1(\R\times\cdots\times\R),\|\cdot\|_*)$ if and only if (\ref{main inequality}) holds. Moreover, in that case,	
	$$\|\H_\varphi\|_{(H^1(\R\times\cdots\times\R),\|\cdot\|_*)\to (H^1(\R\times\cdots\times\R),\|\cdot\|_*)}= \int_0^\infty\cdots \int_0^\infty \varphi(t_1,\ldots,t_n)dt_1\cdots dt_n$$
	and, for any ${\bf e}\in \mathbb E$, $\H_\varphi$ commutes with ${\bf H_e}$ on $\phardy$.
\end{theorem}

In order to prove Theorem \ref{replace by the Hilbert transforms}, we need the following two lemmas.

\begin{lemma}\label{commuting relation to the Hilbert transforms}
	Let $\varphi$ be such that (\ref{main inequality}) holds. Then, for any ${\bf e}\in \mathbb E$, $\H_\varphi$ commutes with the	Hilbert transform ${\bf H_e}$ on $\phardy$.
\end{lemma}

\begin{lemma}\label{key lemma, Hilbert transforms}
	Let $\varphi$ be such that (\ref{main inequality}) holds. Then:
	\begin{enumerate}[\rm (i)]
		\item $\H_\varphi$ is bounded on $(H^1(\pR),\|\cdot\|_*)$, moreover,
		$$\|\H_\varphi\|_{(H^1(\pR),\|\cdot\|_*)\to (H^1(\pR),\|\cdot\|_*)}\leq \pint \pvarphi \pdt.$$
		\item If supp $\varphi\subset [0,1]^n$, then
		$$\|\H_\varphi\|_{(H^1(\pR),\|\cdot\|_*)\to (H^1(\pR),\|\cdot\|_*)}= \int_0^1\cdots \int_0^1 \pvarphi \pdt.$$
	\end{enumerate}
\end{lemma}

\begin{proof}[Proof of Lemma \ref{commuting relation to the Hilbert transforms}]
	Since Theorem \ref{main theorem} and the fact that $H_j$'s are bounded on $\phardy$, it suffices to prove
	\begin{equation}\label{commuting relation to the Hilbert transforms, 1}
		\H_\varphi H_j f= H_j \H_\varphi f
	\end{equation}
	for all $j\in\{1,\ldots,n\}$ and all $f\in\phardy$. Indeed, thanks to the ideas from \cite{An, LM2, LM3} and Lemma \ref{relation to the Fourier transform}(i), for almost every ${\bf y}=(y_1,\ldots,y_n)\in\R^n$,
	\begin{eqnarray*}
		\widehat{\H_\varphi H_j f}({\bf y}) &=& \pint \widehat{H_j f}(t_1 y_1,\ldots,t_n y_n)\pvarphi\pdt\\
		&=& \pint (-i \,\sign(t_j y_j)) \hat{f}(t_1 y_1,\ldots,t_n y_n)\pvarphi\pdt\\
		&=& (-i\, \sign y_j)\widehat{\H_\varphi f}({\bf y})= \widehat{H_j\H_\varphi f}({\bf y}).
	\end{eqnarray*}
	This proves (\ref{commuting relation to the Hilbert transforms, 1}), and thus ends proof of Lemma \ref{commuting relation to the Hilbert transforms}, since the uniqueness of the Fourier transform. 	
\end{proof}

\begin{proof}[Proof of Lemma \ref{key lemma, Hilbert transforms}]
	(i) For all $f\in \phardy$ and all ${\bf e}\in\mathbb E$, by Lemma \ref{commuting relation to the Hilbert transforms} and Theorem \ref{sharp constants for Lp}, we get
	\begin{eqnarray*}
		\|{\bf H_e} \H_\varphi f\|_{L^1(\R^n)} &=& \|\H_\varphi {\bf H_e}f\|_{L^1(\R^n)}\\
		&\leq& \pint\pvarphi\pdt \|{\bf H_e}f\|_{L^1(\R^n)}.
	\end{eqnarray*}
	This proves that
	$$\|\H_\varphi\|_{(H^1(\pR),\|\cdot\|_*)\to (H^1(\pR),\|\cdot\|_*)}\leq \pint \pvarphi \pdt.$$
	
	(ii) The proof is similar to that of  Lemma \ref{key lemma}(ii) and will be omitted. The key point is the estimate (\ref{an useful estimate}) and the fact that $\|\cdot\|_{*}\sim \|\cdot\|_{\phardy}$.

\end{proof}

\begin{proof}[\bf Proof of Theorem \ref{replace by the Hilbert transforms}]
	The proof is similar to that of Theorem \ref{main theorem} by Lemma \ref{key lemma, Hilbert transforms}. We leave the details to the interested readers.
\end{proof}
	
\vskip 0.5cm

{\bf Acknowledgements.}  The authors would like to thank the referees for their
carefully reading and helpful suggestions.

\end{document}